\documentclass[10pt]{amsart}

\usepackage[pdftex]{graphicx}  

\usepackage{multicol}        
\usepackage[bottom]{footmisc}

\makeindex             

\usepackage[utf8x]{inputenc}
\usepackage[english]{babel}
\usepackage{amsmath , amsrefs}
\usepackage{amscd}
\usepackage{amsfonts}
\usepackage{amssymb}
\usepackage{amsthm}
\usepackage{mathrsfs}
\usepackage{graphicx}
\usepackage[all,cmtip]{xy}

\usepackage{todonotes}

\newtheorem{theorem}{Theorem}[section]
\newtheorem{lemma}{Lemma}[section]  
\newtheorem{remark}{Remark}[section]
\newtheorem{proposition}{Proposition}[section]

\newtheorem{example}{Example}[section]
\newtheorem{definition}{Definition}[section]

\newcommand{\Z}{{\mathbb{Z}}}
\newcommand{\C}{{\mathbb{C}}}
\newcommand{\R}{{\mathbb{R}}}
\newcommand{\Q}{{\mathbb{Q}}}
\newcommand{\N}{{\mathbb{N}}}
\newcommand{\KK}{{\mathbb{K}}}
\newcommand{\pp}{{\mathbb{P}}}
\newcommand{\G}{{\mathcal{G}}}

\newcommand{\MS}{{\mathcal{S}}}

\newcommand{\A}{{\mathcal{A}}}
\newcommand{\Pc}{{\mathcal{P}}}
\newcommand{\Nc}{{\mathcal{N}}}
\newcommand{\Ec}{{\mathcal{E}}}
\newcommand{\Kc}{{\mathcal{K}}}

\newcommand{\Sc}{{\mathcal{S}}}
\newcommand{\Cc}{{\mathcal{C}}}
\newcommand{\Bc}{{\mathcal{B}}}
\newcommand{\Fc}{{\mathcal{F}}}
\newcommand{\Tc}{{\mathcal{T}}}
\newcommand{\Dc}{{\mathcal{D}}}

\newcommand{\emme}{{\mathcal M}}

\newcommand{\upi}{\pi}

\newcommand{\ot}{\leftarrow}
\newcommand{\cal}{\mathcal}

\title{On models of the braid arrangement and their hidden symmetries}
\author{Filippo Callegaro, Giovanni Gaiffi}
\date{\today}

\begin{document}

\begin{abstract}
The De Concini-Procesi wonderful  models  of the braid arrangement of type  \(A_{n-1}\)  are equipped with a natural \(S_n\) action, but only the minimal model admits an `hidden'  symmetry, i.e. an action of \(S_{n+1}\) that  comes from its  moduli space interpretation.   
In this paper we explain why    the non minimal models don't admit this extended action:  they  are    `too small'.  In particular we construct  a {\em supermaximal}  model which is  the smallest model that  can be  projected  onto the maximal model and  again admits  an    extended \(S_{n+1}\) action. 
We give an explicit description of a basis for the integer cohomology of this supermaximal model.

Furthermore, we deal with another    hidden extended action of the symmetric group: we  observe that the symmetric group \(S_{n+k}\) acts by permutation on  the set of \(k\)-codimensionl strata of the minimal model.  Even if this  happens at a purely combinatorial level,  it gives rise to an interesting permutation action on the  elements of a basis of the integer cohomology. 
\end{abstract}
\maketitle

\section{Introduction}
\label{sec:1}

In this paper we focus on two different `hidden' extended actions of the symmetric group on  wonderful models of the (real or complexified) braid arrangement.  
As it is well known, there  are several  De Concini-Procesi  models associated with  the  arrangement of type \(A_{n-1}\) (see \cite{DCP1}, \cite{DCP2}); these are smooth varieties, proper over the complement of the arrangement,  in which the union of the subspaces is replaced by a divisor with normal crossings.  Among these spaces there is a minimal one (i.e. there are birational projections from the  other spaces onto it), and a maximal one (i.e. there are birational projections from it onto the other spaces).   
The natural \(S_n\) action on the complement of the  arrangement  of type \(A_{n-1}\) 
extends to all of these models. 

The first of the two extended actions which we deal with is well known and comes from the following remark:  the minimal projective (real or complex) De Concini-Procesi model  of type \(A_{n-1}\) is isomorphic to the moduli space \({\overline M_{0,n+1}}\) of \(n+1\)-pointed stable  curves of genus 0, therefore it carries an `hidden' extended  action of \(S_{n+1}\) that has been  studied by several authors (see for instance \cite{getzler},  \cite{rains2009}, \cite{etihenkamrai}).

Now we  observe  that the  \(S_{n+1}\) action cannot be extended to the non-minimal  models (we show this by an example in Section \ref{sec:supermaximal}).   

Why does this happen?  This is  the first problem  discussed in the present  paper. 
We answer to this question  by showing  in Section \ref{sec:chiusuraazione}   that the maximal model is, in a sense, `too small'. This takes two steps (see Theorem \ref{teo:chiusoazionebuilding}):
\begin{enumerate}
\item  we identify in a natural way its strata with a subset \(\mathcal {T}\)  of 1-codimensional  strata of  a `supermaximal' model on which  the \(S_{n+1}\) action is defined. This  supermaximal model is obtained by blowing up some strata in the maximal model, but it also  belongs to the family \(\mathcal{L}\) of models obtained by blowing up {\em building sets} of strata in the minimal model;   in fact it is the model obtained by blowing up all the strata of the minimal model.  The models in \(\mathcal{L}\) are examples  of some well known  constructions that, starting from a `good' stratified variety, produce models by blowing up a suitable subset of  strata (see \cite{procesimacpherson},  \cite{li}, \cite{gaimrn0} and also \cite{denham} for further references);

\item we show that the closure of \(\mathcal {T}\)  under the \(S_{n+1}\) action is the  set of all the strata of the supermaximal model. More precisely, this  means that   the supermaximal model is the minimal model in \(\mathcal{L}\)  that admits a birational projection onto  the maximal model  and is equipped with  the  \(S_{n+1}\) action.

\end{enumerate}

The second problem addressed by this paper is the computation of the integer cohomology  of the  complex supermaximal modes described above. The cohomology module   provides `geometric' extended representations of \(S_{n+1}\) and in Theorem  \ref{thm:cohomology} we exhibit an explicit  basis for it.    Actually, the statement of Theorem  \ref{thm:cohomology} is much more general: given any complex subspace arrangement we consider its minimal De Concini-Procesi model and  we describe a basis for the integer cohomology of the variety obtained by blowing up all the strata in  this minimal model (this variety generalizes in a way the notion of  supermaximal model).

We then compute a generating formula for the  Poincar\'e polynomials of the complex supermaximal models of braid arrangements (see Theorems \ref{teo:sostituzioni}, \ref{teo:formulapoincare}). As a consequence, we  also give a formula for the  Euler characteristic series in the real case, where Euler secant numbers appear (see Corollary \ref{teo:eulerchar}).

In the last two sections, Section \ref{sec:combinatorialaction} and Section \ref{ss:action3},  we show that there is   another    hidden extended action of the symmetric group on the minimal model  of a  braid arrangement,  that is  different from the action  described above. 

In fact, motivated by a combinatorial remark proven in \cite{gaifficayley}, we  observe that the symmetric group \(S_{n+k}\) acts by permutation on  the set of \(k\)-codimensionl strata of the minimal model of  type \(A_{n-1}\). 

This happens at a purely combinatorial level  and it does not correspond to a geometric  action on the minimal model, nevertheless it gives rise to an interesting permutation action on the  elements of a basis of the integer cohomology of the complex minimal model. 
The splitting of these elements  into orbits
allows us  to write (see Theorem \ref{teo:formulapoicareminimale})  a generating formula for  the Poincar\'e polynomials of the complex minimal models  that is different from the ones available in the literature (see for instance the recursive formula for the Poincar\'e series computed, in three different ways, in  \cite{Manin}, \cite{YuzBasi}, \cite{GaiffiBlowups}).


\section{Wonderful Models}\label{wonderful}
\subsection{The Geometric Definition of Building Sets and Nested Sets}
    \label{some}
     In this section we recall from \cite{DCP1}, \cite{DCP2} 
     the  definitions of building set and nested set of subspaces. 
    Let $V$ be a real or complex finite dimensional vector space 
    endowed with an Euclidean or Hermitian non-degenerate product
    and let  $\A$ be a central subspace   arrangement in  $V$.   
    For every $A\in V$, we will  denote by  $A^\perp$ its
orthogonal. 
We denote   by ${\mathcal C}^{}_{\A}$ the closure  under the sum of $\A^{}$ and 
by $\A^{\perp}$ the arrangement of subspaces in \(V\)
\[\A^\perp=\{A^\perp \:|\: A\in \A\}.\]

\begin{definition}
The  collection of subspaces $\G \subset {\mathcal C}^{}_{\A}$ is called  {\em building set associated to \(\A\)} if \({\mathcal C}^{}_{\A}= {\mathcal C}^{}_{\G}\) and  every element $C$ of ${\mathcal C}^{}_{\A}$ is the direct sum 
 $C= G_1\oplus G_2\oplus \ldots\oplus G_k$ of  the maximal elements 
$G_1,G_2,\ldots,G_k$  of
$\G^{}$ contained in $C$ (this is called the \(\G\)-decomposition of \(C\)).
\end{definition}

Given a subspace  arrangement  \(\A\), there are several building sets associated to it. Among these there always are a maximum and a minimum (with respect to inclusion). The maximum is ${\mathcal C}^{}_{\A}$, the minimum is the building set of {\em irreducibles} that is defined as follows.

\begin{definition}
 Given a subspace $U\in\Cc_\A$, a \textbf{decomposition of} $U$ in $\mathbf{\Cc_\A}$ is a collection
$\{U_1,\ldots,U_k\}$ ($k>1$) of non zero subspaces in $\Cc_\A$ such that
\begin{enumerate}
 \item $U=U_1\oplus\cdots\oplus U_k;$
 \item for every subspace $A\in\Cc_\A$ such that $A\subset U$, we have $A\cap U_1,\ldots,A\cap U_k \in \Cc_\A$ and
$A=\left(A\cap U_1\right)\oplus\cdots\oplus \left(A\cap U_k\right)$.
\end{enumerate}
\end{definition}
\begin{definition}
 A nonzero subspace $F\in\Cc_\A$ which does not admit a decomposition  is called \textbf{irreducible} and the set of
irreducible subspaces is denoted by $\mathbf{\Fc_\A}$.
\end{definition}
\begin{remark}
\label{rootirreducibles}
As an example, let us consider  a root system \(\Phi\) in \(V\) (real or complexified vector space) and its associated 
 root  arrangement (i.e.  $\A^\perp$ is the hyperplane arrangement provided by the hyperplanes orthogonal to   the roots in  \(\Phi\)). In this case  the  building set of irreducibles is the set  whose elements are  the  subspaces spanned by  the irreducible root subsystems of \(\Phi\) (see \cite{YuzBasi}).

\end{remark}



 \begin{definition}
Let $\G$ be a building set associated to \(\A\). 
A subset $\MS\subset \G$ is called 
($\G$-)nested, if
 given any subset $\{U_1,\ldots ,U_h \}\subseteq \MS$ (with \(h>1\)) of pairwise non comparable 
	elements, we have that $U_1 + \cdots + U_h\notin \G$.
\end{definition}

\subsection{The Example of the Root System \(A_{n-1}\)}
\label{secexamplean}

 Let \(V=\R^n\) or \(\C^n\) and let us consider  the real or complexified root arrangement of type \(A_{n-1}\).  We think of it  as an {\em essential} arrangement, i.e. we consider the hyperplanes defined by the equations \(x_i-x_j=0\) in the quotient space \(V/<(1,1,...,1)>\).

Let us denote by  $\Fc_{A_{n-1}}$ the building set of irreducibles associated to this arrangement. According to Remark \ref{rootirreducibles}, it is made  by all the subspaces in \(V\) spanned by the irreducible root subsystems.  Therefore there  is a bijective correspondence between the elements of $\Fc_{A_{n-1}}$ and the subsets of $\{1,\cdots,n\}$ of cardinality at
least two:  if the orthogonal of \(A\in \Fc_{A_{n-1}}\) is the subspace  described by the equation  \(x_{i_1}=x_{i_2}= \cdots =x_{i_k}\) then we represent \(A\) by the set \(\{i_1,i_2,\ldots, i_k\}\).
As a consequence,   a $\Fc_{A_{n-1}}$-nested set  \(\Sc\) is  represented by a set  (which we still call \(\Sc\)) of subsets of $\{1,\cdots,n\}$ with the property that any of its elements has  cardinality \(\geq 2\) and if \(I\) and \(J\) belong to \(\Sc\) than either \(I\cap J=\emptyset\) or one of the two sets is included into the other.

As an example we can consider  the  $\Fc_{{A_{8}}}$-nested set represented  by the  three sets \(\{1,5\}\), \(\{2,4,6\}\), \(\{2,4,6,7,8\}\). This means that the elements of the nested set are the  three subspaces whose orthogonal subspaces in  \(V/<(1,1,...,1)>\) are described respectively by the equations \(x_{1}=x_{5}\), \(x_{2}=x_{4} =x_{6}\) and 
\(x_{2}=x_{4}= x_{6}=x_{7}=x_{8}\).

We observe that we  can  represent a $\Fc_{A_{n-1}}$-nested set   \(\MS\) 
 by an oriented   forest on \(n\) leaves  in the following way. We consider the set \({\tilde \MS}=\MS\cup\{1\}\cup \{2\}\cup\cdots \cup \{n\}\).
Then the forest  coincides  with the Hasse diagram of  \({\tilde \MS}\) viewed as a poset by the  inclusion relation:  the roots of the trees correspond to  the maximal elements of \(\MS\), and the orientation goes from the roots to the leaves, that are the vertices \(\{1\},\{2\},\ldots, \{n\} \).

Let us now focus on the maximal building set \(\Cc_{A_{n-1}}\) associated with the root arrangement of type \(A_{n-1}\). It is made by all the subspaces that can be obtained as the span of a set of roots. Using the same notation as before, these subspaces can be put in bijective correspondence with the partitions of $\{1,\cdots,n\}$ such that at least one part has cardinality \(\geq 2\). 

For instance, 
\[   \{1,4\},\{2\},\{3,5,9\},\{6\},\{7,8\}\]
corresponds to the subspace whose orthogonal is described by the equations \(x_{1}=x_{4}\) and  \(x_{3}=x_{5} =x_{9}\) and 
\(x_{7}=x_{8}\).



The \(\Cc_{A_{n-1}}\)-nested sets  are given by chains of subspaces in \(\Cc_{A_{n-1}}\) (with respect to   inclusion). In terms of partitions, this corresponds to give chains of the above described partitions of $\{1,\cdots,n\}$ (with respect to the refinement relation).


One can find in   \cite{GaiffiServenti} a description of the maximal model \(Y_{\Cc_{{A_{n-1}}}}\) and in \cite{GaiffiServenti2} a description of all the \(S_n\)-invariant building sets associated with the root system \(A_{n-1}\).

\subsection{The construction of wonderful models and their cohomology}
  \label{constru}
In this section we  recall  from   \cite{DCP1} the  construction and the main properties of the De Concini-Procesi models.

 The  interest in these models  was at first motivated by an  approach to Drinfeld's construction of special
solutions for Khniznik-Zamolodchikov equation (see \cite{drinfeld}).  Moreover, in \cite{DCP1} it was shown, using the cohomology description of these models, that  the rational homotopy type of the complement of a complex  subspace arrangement  depends only on the intersection lattice.

 Then real and complex  De Concini-Procesi models
turned out to play a key role in several   fields of mathematical research:  subspace and toric  arrangements,  toric varieties (see for instance \cite{DCP3}, \cite{feichtneryuz}, \cite{rains}),   tropical geometry (see \cite{feichtnersturmfels}), 
moduli spaces  and       configuration spaces (see for instance \cite{etihenkamrai}, \cite{LTV}), box splines, vector partition functions  and  index theory (see  \cite{DCP4}, \cite{cavazzanimoci}),  discrete  geometry    (see \cite{feichtner}).

 Let us recall how they are defined. We will  focus on the case when \(\KK=\R\) or \(\KK=\C\).
 Let \(\A\) be a subspace arrangement in the real or complex space  \(V\) 
 endowed with a non-degenerate euclidean or Hermitian product
  and let \(\emme(\A^\perp)\) be the  complement in \(V\) of the arrangement \(\A^\perp\).
  Let $\G$ be a building set associated to \(\A\) (we can suppose that it contains $V$). 
 Then one considers   the map
\begin{displaymath}
	 i\: :\: \pp( \emme (\A^\perp)  ) \rightarrow  \pp(V) \times \prod 
_{D\in
	\G-\{V\}}^{}\pp(V /D^\perp)
\end{displaymath}
where in the first coordinate we have the inclusion and the map from
$\emme (\A^\perp) $ to $\pp(V /D^\perp )$
 is the restriction of the
canonical projection ${\displaystyle (V-D^\perp) \rightarrow 
\pp(V /D^\perp)}$.
\begin{definition}
The  (compact) wonderful  model
$ Y_{\G}$ is  obtained
by  taking the closure of the image of $ i$.
\end{definition}
De Concini and Procesi in \cite{DCP1} proved  that
  the complement $ { \mathcal D}$ of  $\pp( \emme (\A^\perp) ) $ in $
Y_{\G}$ is a divisor with normal crossings whose irreducible components are in bijective correspondence with the elements of \(\G-\{V\}\) and are denoted by \(\Dc_G\)  
($G\in
	\G-\{V\}$). 

	



If we denote by $\upi$ the 
projection of $
Y_{\G}$ onto
	the first component \(  \pp(V)\),  then ${ \mathcal  D}_G$		can  be characterized as the unique 
irreducible component such that
	$ \upi({ \mathcal  D}_G)=\pp(G^\perp)$.  
	
	A complete characterization of the boundary is then provided by the 
		observation that, if we consider  a collection \( 
	{ \mathcal T} \) of subspaces in \( \G \) containing \(V\),  then  \[ 
	{ \mathcal D}_{{ \mathcal T}}= \bigcap_{A\in { \mathcal 
	T}-\{V\}}{\mathcal D}_{A}\] is non empty if and only if \( {\mathcal T} \) 
	is \(\G\)-nested, and in this case \(	{\mathcal D}_{{\mathcal T}}\)  is a smooth irreducible subvariety obtained as a normal crossing intersection.
Sometimes we will denote $Y_{\G}$ by  
 ${ \mathcal D}_\MS$ with $\MS=\{V\}$.


The integer cohomology ring of the models \(Y_\G\)  in the complex case was studied  in  \cite{DCP1}, where  a  presentation by generators and relations  was provided. The cohomology is   torsion free, and in  \cite{YuzBasi}  Yuzvinski explicitly described  some $\Z$-bases (see also  \cite{GaiffiBlowups}).  We  briefly  recall these results (for a description of the cohomology ring in the real case see \cite{rains}).

Let \(\G\) be a building set containing \(V\) and let us consider a $\G$-nested set $\MS$. 
We  take a 
subset ${\cal H}\subset \G$  and an  element $B \in \G$ 
with the property that $A\subsetneq B$ for all $A\in {\cal H}$ 
and  we put  $\MS_B=\{A\in \MS\: :\: 
A\subsetneq B\}$. As in \cite{DCP1}, we define the non negative integer 
$	d_{{\cal H},B}^\MS$:
\begin{definition}
\begin{displaymath}
	d_{{\cal H},B}^\MS=dim B-dim 
	\left( \sum _{A\in {\cal H}\cup \MS_B}^{}A \right)
\end{displaymath}
\end{definition}
Then we consider the polynomial ring $\Z[c_A]$ where the variables 
$c_A$ are  indexed by the elements of $\G$.  
\begin{definition} Given \(\G\), $\MS$, ${\cal H}$  and  $B$ as before, we define the following polynomial in $\Z[c_A]$: 
\begin{displaymath}
	P_{{\cal H},B}^\MS=\left(\prod _{A\in {\cal H}}^{}c_A \right) \left(\sum _{B\subset 
	C}^{}c_C \right)^{d_{{\cal H},B}^\MS}
\end{displaymath}
\end{definition}
Let us now call by $I_\MS$  the ideal in $\Z[c_A]$ generated by these 
polynomials, for fixed $\MS$ and varying ${\cal H}, B$.
\begin{theorem}[{see \cite[Section 5.2]{DCP1}}]
Let $\G$  and   $\MS$ be as before, and let us consider the complex model \(Y_\G\). The natural map $\phi \: : \Z[c_A] \mapsto H^{\ast}(\Dc_\MS,\Z)$, defined by sending 
$c_A$ to the cohomology class 
$[\Dc_A]$ associated to the divisor $\Dc_A$ (restricted to $\Dc_S$), induces an 
isomorphism  between ${\displaystyle \Z[c_A]/I_\MS}$ and $H^{\ast}(\Dc_\MS,\Z)$.
In particular,  in the case 
$\MS=\{V\}$, we obtain 
\begin{displaymath}
\Z[c_A]/I_{\{V\}}	\simeq H^{\ast}(Y_{\G},\Z)
\end{displaymath}
\label{coom}
\end{theorem}

\begin{definition}

 Let $\G$  and   $\MS$ be as before. A function $f\: :\: \G \: \mapsto \N $ is 
called $\G,\MS$-admissible if it is $f=0$ or if $f\neq 0$, $supp\, f\cup \MS$ is 
$\G$-nested and, for every $A\in supp\, f$, 
$f(A)<d_{(supp\, f)_A,A}^{\MS}$.
\end{definition}

Now, given a $\G,\MS$-admissible function $f$, we can consider  in 
$H^{\ast}({ \mathcal D}_\MS,\Z)\simeq \Z[c_A]/I_{\MS}$ the monomial 
${\displaystyle m_f=\prod _{A\in \G}^{}c_A^{f(A)}}$. We will call 
``$\G,\MS$-admissible'' such monomials.

\begin{theorem}[{see \cite[Section 3]{YuzBasi} and \cite[Section 2]{GaiffiBlowups}}]\label{base coomologia}
The set ${\cal B}_{\G,\MS}$ of $\G,\MS$-admissible monomials is a $\Z$-basis for 
$ H^{\ast}({ \mathcal D}_\MS,\Z) $.
\end{theorem}

\subsection{A More General Construction}
\label{secgeneralblowup}

The construction of De Concini-Procesi models  can be viewed as a  special case of  other more general constructions that, starting from a `good' stratified variety, produce models by blowing up a suitable subset of  strata. Among these there are  the models  described by MacPherson and Procesi in \cite{procesimacpherson} and by  Li in \cite{li}.
 In Li's paper one can also find a  comparison among several  constructions of wonderful compactifications by  Fulton-Machperson (\cite{fultonmacpherson}),  Ulyanov (\cite{Ulyanov}), Kuperberg-Thurston (\cite{Kuperbergthurston}), Hu (\cite{Hu}). A further interesting survey including tropical compactifications can be found in Denham's paper \cite{denham}.

We recall here some basic facts adopting the language  and the notation of Li's paper.

\begin{definition}\label{def:simple} A simple arrangement of subvarieties (or `simple stratification') of a nonsingular variety \(Y\) is a finite set \(\Lambda = \{\Lambda_i\}\) of nonsingular closed subvarieties \(\Lambda_i\) properly contained in \(Y\) satisfying the following conditions:\\
(i) the  intersection of \(\Lambda_i\) and \(\Lambda_j\)  is nonsingular and the tangent bundles satisfy \(T(\Lambda_i\cap \Lambda_j) = T(\Lambda_i)_{|(\Lambda_i\cap \Lambda_j)}\cap T(\Lambda_j)_{|(\Lambda_i\cap \Lambda_j)}\),\\
(ii) \(\Lambda_i \cap \Lambda_j\) either is equal to some stratum in \(\Lambda\)  or is empty.

\end{definition}

\begin{definition}
Let \(\Lambda\) be an arrangement of subvarieties of  \(Y\).  A subset \(\G' \subseteq \Lambda\) is called a building set of \(\Lambda\)  if \(\forall \Lambda_i \in \Lambda - \G'\)  the minimal elements in \(\{G  \in \G' \: : \: G \supseteq \Lambda_i\}\)  intersect transversally and the intersection is \(\Lambda_i\).

\end{definition}

Then, if one has a simple stratification  \(\Lambda\) of a nonsingular variety \(Y\) and a  building set \(\G'\), one can construct a wonderful model \(Y_{\G'}\)   by  considering  (by  analogy with \cite{DCP1})  the closure of the image of the  locally closed embedding
\[\left( Y-\bigcup_{\Lambda_i\in \Lambda}\Lambda_i \right ) \rightarrow \prod_{G\in \G'}Bl_GY\] where \(Bl_GY\) is the blowup of \(Y\) along \(G\).

It turns out that 

\begin{theorem}[{see \cite[Theorem 1.3]{li}}]
\label{teo:listabuilding}
If one arranges the elements \(G_1,G_2,...,G_N\)  of \(\G'\) in such a way that for every \(1\leq i \leq N\) the set \(\{G_1,G_2,\ldots , G_i\}\) is building, then \(Y_{\G'}\) is isomorphic to the variety  \[Bl_{{\widetilde G_N}}Bl_{{\widetilde G_{N-1}}}\
\cdots Bl_{{\widetilde G_2}}Bl_{G_1}Y\]
where \({\widetilde G_i}\) denotes the dominant transform of $G_i$ in $Bl_{{\widetilde G_{i-1}}}
\cdots Bl_{{\widetilde G_2}}Bl_{G_1}Y$.
\end{theorem}


\begin{remark}
\label{rem:ordinescoppiamenti}
As remarked by Procesi-MacPherson in \cite[Section 2.4]{procesimacpherson} it is
always possible to choose a linear ordering on the set $\G'$ such that every initial segment is building. We can do this by ordering $\G'$ in such a way that we always blow up first the strata of smaller  dimension.
\end{remark}

We show two examples that will be crucial in the following sections.

\begin{example}
In the case of    subspace arrangements, the De Concini-Procesi construction  
and the above construction produce the same models  (the only warning is that in the preceding sections a building set \(\G\) was described in a dual way,  so the building set of subvarieties \(\G'\) is made by  the  orthogonals of the subspaces in  \(\G\)).

\end{example}

\begin{example}\label{es:boundary}
Given a  De Concini-Procesi   model \(Y_\G\), we notice that its boundary strata   give rise to a simple arrangement of subvarieties, and that the set of all strata is a building set. So it is possibile to  obtain a `model of the model \(Y_\G\)'.  
The boundary strata of these `models of models' are indexed by the nested sets of the building set of all the strata of \(Y_\G\).  More precisely, according to the definition given in \cite[Section 4]{procesimacpherson}, a nested set \(\Sc\) in this sense is a collection of \(\G\)-nested sets containing $\{V\}$ linearly ordered by inclusion
(we will come back to this, using a more combinatorial definition,  in Section \ref{sec:buildingcombinatorial}). 
\end{example}

%

\section{Combinatorial Building Sets}
\label{sec:buildingcombinatorial}
After De Concini and Procesi's paper \cite{DCP1},    nested sets and building sets appeared in the literature, connected with several combinatorial problems.  
In   \cite{feichtnerkozlovincidence}   building sets and nested sets  were defined in the  general context of meet-semilattices, and in   \cite{delucchinested} their  connection with Dowling lattices was investigated.  Other purely combinatorial definitions were used   to  give rise to  the polytopes that  were named {\em nestohedra} in \cite{postnikoreinewilli}. 

Here we recall  the combinatorial definitions of building sets and nested sets of a power set   in the spirit of \cite{postnikoreinewilli}, \cite{postnikov} (one can  refer to \cite[Section 2]{petric2} for a short comparison among various definitions and notations in the literature). 

\begin{definition}
\label{def:buildingpowerset}
A building set   of the power set \(\Pc(\{1,2,...,n\})\) is a subset \({\mathcal B}\) of \(\Pc(\{1,2,...,n\})\) such that:
\begin{itemize}
\item[a)] If \(A,B \in {\mathcal B}\) have nonempty intersection, then \(A\cup B\in {\mathcal B}\).
\item[b)]The set \(\{i\}\) belongs to \({\mathcal B}\) for every \(i\in \{1,2,...,n\}\).
\end{itemize}

\end{definition}

\begin{definition}
\label{def:nestedpowerset}
A (nonempty) subset  \({\mathcal S}\) of a building set  \({\mathcal B}\) is a $\mathcal B$-nested set (or just nested set if the context is understood) if and only if the following two conditions hold: 
\begin{itemize}
\item[a)] For any \(I,J\in {\mathcal S}\) we have that either \(I\subset J\) or \(J\subset I\) or \(I\cap J=\emptyset\).
\item[b)] Given elements  \(\{J_1,...,J_k\}\) (\(k\geq 2\)) of \({\mathcal S}\) pairwise not comparable with respect to inclusion, their union is not in \({\mathcal B}\). 
\end{itemize}

\end{definition}


\begin{definition}The nested set complex  \({\mathcal N}({\mathcal B})\) is the poset  of all  the nested sets of \({\mathcal B}\) ordered by inclusion. 
\end{definition}


We notice  that actually \({\mathcal N}({\mathcal B})\cup \{\emptyset \}\)  is a simplicial complex.

\begin{definition}
If the set \({\mathcal B}\) has a minimum $\mu$, the nested set complex  \({\mathcal N}'({\mathcal B})\) is the poset  of all  the nested sets of \({\mathcal B}\) containing $\mu$, ordered by inclusion. 
\end{definition}

In particular, let us denote by  $\Bc(n-1)$ 
the poset ${\mathcal N}'(\Fc_{A_{n-1}})$ given
by the nested sets in \(\Fc_{A_{n-1}}\) that contain $\{V\}$. 

We observe that any element in ${\mathcal N'}(\Fc_{A_{n-1}})$ can be obtained by the union of $\{V\}$ with an element of  
\[ \Pc(\A_1')\cup \Pc(\A_2')\cup...\cup \Pc(\A_s')\]
where \(\A_j' = \A_j- \{V\}\) and $\A_j$ are all the maximal nested sets associated with the building set \(\Fc_{A_{n-1}}\) (and \(\Pc(\ ) \)   denotes the power set).
Given a simplicial complex \(\Cc\) which is based on some sets \(\A_1',..,\A_s'\) (i.e., it is equal to 
\(\Pc(\A_1')\cup \Pc(\A_2')\cup...\cup \Pc(\A_s')\)),  Feichtner and Kozlov's  definition  of building set of a meet semilattice (see \cite[Section 2]{feichtnerkozlovincidence}) can be expressed    in the following way: \(\Bc\subseteq \Cc\) is a building set of \(\Cc\) if and only if for every \(j=1,2,...,s\) the set \(\Bc\cap \Pc(\A_j')\) is a building set of \(\Pc(\A_j')\) in the sense of Definition \ref{def:buildingpowerset}.

Again, according to Feichtner and Kozlov,  given a building set \(\Bc\) of \(\Cc\) as before, a \(\Bc\)-nested set is a subset \(\Sc\) of \(\Bc\) 
such that, for every antichain (with respect to inclusion) \(\{X_1,X_2,...,X_l\}\subseteq \Sc\), the union \(X_1\cup X_2\cup ..\cup X_l\) belongs to \(\Cc-\Bc\).

These definitions of building set and nested sets can be extended in a natural way to  \(\Bc(n-1)={\mathcal N}'(\Fc_{A_{n-1}})\) . In particular,  the   maximal building set of \(\Bc(n-1)\) is \(\Bc(n-1)\) itself. 
As we observed in Section \ref{constru}, the strata of \(Y_{\Fc_{A_{n-1}}}\) are indexed by the elements of  $\Bc(n-1)$ and, 
as we remarked in Section \ref{secgeneralblowup}, the set of all these strata is a  building set in the sense\footnote{According to the original definition, the building set is  $\Bc(n-1)-\{\{V\}\}$, but one can immediately check that one can add \(\{V\}\) without affecting the construction.}  of  \cite{procesimacpherson} \cite{li} and \cite{gaimrn0}, therefore we can construct the corresponding (real or complex) variety 
\(Y_{\Bc(n-1)}\).

Translating into these combinatorial terms the definition given in \cite[Section 4]{procesimacpherson},
the strata  of the  variety \(Y_{\Bc(n-1)}\) 
are indexed by the nested sets of $\Bc(n-1)$ containing \(\{V\}\) in the following way: 
a stratum of codimension \(r\) is indexed by  \(\{\{V\}, \Tc_1,..., \Tc_r\}\)  where each \(\Tc_i\) belongs to  \(
\Bc(n-1)
\) and 
\[\{V\}\subsetneq  \Tc_1\subsetneq \cdots \subsetneq \Tc_r.\]

%

\section{The geometric extended \(S_{n+1}\) action on $Y_{\Fc_{A_{n-1}}}$} \label{sec:action1}

We recall that there  is a well know `extended' \(S_{n+1}\) action on the De Concini-Procesi (real or complex)  model \(Y_{\Fc_{A_{n-1}}}\):  it comes from the isomorphism with the moduli space \({\overline M_{0,n+1}}\) 
and the character of the resulting representation  on cohomology has been computed in \cite{rains2009} in the real case, and in \cite{getzler} in the complex case.

In order to describe how \(S_{n+1}\) acts on the strata   of \(Y_{\Fc_{A_{n-1}}}\) it is sufficient to show the corresponding action on \(\Fc_{A_{n-1}}\), since these strata are indexed by the elements of \(\Bc(n-1)\).

Let \(\Delta=\{\alpha_0,\alpha_1,...,\alpha_{n-1}\}\) be a basis for the root system \(\Phi_{A_n}\) of type \(A_n\) (we added to a basis of \(A_{n-1}\) the extra root \(\alpha_0\)). 
We identify in the standard way  \(S_{n+1}\) with the group which permutes \(\{0,1,...,n\}\) and    \(s_{\alpha_0}\) with the transposition \((0,1)\). Therefore    \(S_{n}\), the subgroup generated by \(\{s_{\alpha_1},...,s_{\alpha_{n-1}}\}\), is identified with the subgroup which permutes \(\{1,...,n\}\).

 Let \(A\) be a subspace in \(\Fc_{A_{n-1}}\) different from \(V\),  let \(\sigma \in S_{n+1}\) and let us consider the subspace \(\sigma A\) according to the natural action of \(S_{n+1}\) on  \(\Fc_{A_{n}}\). Morover  we   denote  by \(\overline A\) the subspace generated by all the roots of \(\Phi_{A_n}\)   that are orthogonal to \(A\). We notice that if 
some of the roots contained in \(\sigma A\) have  \(\alpha_0\) in their support 
 then \(\sigma \overline A\) belongs to \(\Fc_{A_{n-1}}\).
 Therefore we define the action of $S_{n+1}$ on  \(\Fc_{A_{n-1}}\) as follows. Let $\sigma \in S_{n+1}$ and 
 \(A \in \Fc_{A_{n-1}}\).
 We set \(\sigma \cdot V:=V\) and for $A \neq V$ we define $$\sigma \cdot A := \left\{ 
 \begin{array}{ll}
 \sigma A & \mbox{ if }\sigma A \in  \Fc_{A_{n-1}} \\
  \sigma \overline A & \mbox{ otherwise.}
 \end{array}
 \right.
 $$

Let us now write explicitly    how the above described  action extends to  \(\Bc(n-1)\).  Let \[\Sc=\{V, A_1,A_2,...,A_k, B_1,B_2,...,B_s\}\] be an element of  \(\Bc(n-1)\), i.e. a nested set  in \({\Fc_{A_{n-1}}}\)  that contains \(V\) and let \(\sigma \in S_{n+1}\).  
Moreover, let us suppose that, for every subspace \( A_j\),  the subspace \(\sigma A_j\) doesn't belong to \(\Fc_{A_{n-1}}\), while  the subspaces 
\(\sigma B_t\) belong to \(\Fc_{A_{n-1}}\). Then \(\sigma \Sc= \{V, ..,\sigma\overline{A_j},.., \sigma B_s,..\}\). 
As one can quickly check, \(\sigma \Sc \in \Bc(n-1)\). 
\begin{remark}This  action can also be lifted to the minimal spherical model of type \(A_{n-1}\)  (see for instance the exposition in \cite[Section 3]{callegarogaiffi3}).
\end{remark}

%


\section{The Extended Action on Bigger Models: the Example of \(A_3\)}
\label{sec:supermaximal}
From now on the minimal and the maximal models associated with the root system $A_{n-1}$ will play a special role in this paper. Hence it is convenient to single out them by a new notation. 
\begin{definition}
We will denote by \(Y_{minA_{n-1}}\) the minimal model  \(Y_{\Fc_{A_{n-1}}}\) 
and by \(Y_{maxA_{n-1}}\) the maximal model \(Y_{\Cc_{A_{n-1}}}.\)
\end{definition}
It is known that is not possible to extend the \(S_{n+1}\) action from the 
strata of the boundary
of  \(Y_{minA_{n-1}}\) to the 
strata of the boundary  of the non minimal models
 (see for instance   \cite[Remark 5.4]{GaiffiServenti2}).

Now we want to construct  a model which is `bigger' than \(Y_{maxA_{n-1}}\) (i.e. it admits a birational projection onto \(Y_{maxA_{n-1}}\)) and is equipped with an   \(S_{n+1}\) action.  We will call {\em supermaximal}  a  model which is minimal  among the models that have  these properties.   
Let us construct this by an  example in the case \(Y_{maxA_{3}}\).

\begin{example} 
We consider the action of the group $S_5$ on the  model \(Y_{minA_{3}}\):
the transposition $s_{\alpha_0}$ maps the $1$-dimensional strata $\{V, <\alpha_1>\}$ and $\{V, <\alpha_3>\}$ as follows:
$$
s_{\alpha_0}\{V, <\alpha_1>\}= \{V, s_{\alpha_0}\overline{<\alpha_1>}\}=\{V, <\alpha_1+\alpha_2,\alpha_3>\}
$$
and
$$
s_{\alpha_0}\{V, <\alpha_3>\}=\{V, s_{\alpha_0}<\alpha_3>\}= \{V, <\alpha_3>\}.
$$

As one of the steps in the construction of  $Y_{maxA_{3}}$,   we blow up $Y_{minA_{3}}$ along  the intersection $$\{V, <\alpha_1>\} \cap \{V, <\alpha_3>\} = \{V,  <\alpha_1>, <\alpha_3>\}.$$ Hence, in order to have a model with an extended  $S_5$ action, we also need  to blow up $Y_{maxA_{3}}$ along the intersection $$\{V, <\alpha_1+\alpha_2,\alpha_3>\} \cap \{V, <\alpha_3>\} = \{V, <\alpha_1+\alpha_2,\alpha_3>, <\alpha_3>\}.$$
Actually, because of the \(S_4\) symmetry, one has to blow up in $Y_{maxA_{3}}$ all the points  $\{V, <\alpha,\beta>, <\alpha> \}$, where  the two roots \(\alpha, \beta\) span an irreducible root subsystem.

Let us denote by \(Y_{supermaxA_{3}}\) the model obtained as the  result of  all these blowups:  one can immediately check that \(Y_{supermaxA_{3}}\) coincides with \(Y_{\Bc(3)}\) and therefore the   \(S_{5}\)  action on  \(Y_{minA_{3}}\) described in Section 
\ref{sec:action1} 
extends to  $Y_{supermaxA_{3}}$. In the next two   sections we will prove that   \(Y_{\Bc(n-1)}\) is a supermaximal model  for every \(n\geq 3\).
\end{example}
%

\section{The \(S_{n+1}\) action on \(Y_{\Bc(n-1)}\) and the minimality property of  \(\Bc(n-1)\)} 
\label{sec:chiusuraazione}

As explained in Section \ref{sec:buildingcombinatorial} the strata of \(Y_{\Bc(n-1)}\) are in correspondence with the elements of \({\mathcal N}'(\Bc(n-1) )\).  The \(S_{n+1}\) action on the open part of \(Y_{\Bc(n-1)}\) can be extended to the boundary. 
In fact one  can immediately check that the   geometric    \(S_{n+1}\) action on \(\Bc(n-1) \) can be extended to \({\mathcal N}'(\Bc(n-1) )\): let \(\{\{V\}, \Tc_1,..., \Tc_r\}\) be and element of \({\mathcal N}'(\Bc(n-1) )\) and let \(\sigma \in S_{n+1}\), then \(\sigma\) sends \(\{\{V\}, \Tc_1,..., \Tc_r\}\) to \(\{\{V\}, \Tc_1',..., \Tc_r'\}\) where, 
for every \(i\), $\Tc_i'=\sigma \Tc_i$  according to the action on \(\Bc(n-1) \) illustrated in the end of Section \ref{sec:action1}. From the inclusions $\{V\}\subsetneq  \Tc_1\subsetneq \cdots \subsetneq \Tc_r$
it immediately follows that $\{V\}\subsetneq  \Tc_1'\subsetneq \cdots \subsetneq \Tc_r'$, therefore \(\{\{V\}, \Tc_1',..., \Tc_r'\}\) belongs to \({\mathcal N}'(\Bc(n-1) )\).

Now we address the following combinatorial problem:  what is the minimal building set in   \(\Bc(n-1) \) that is closed under the \(S_{n+1}\) action and `contains' \(\Cc_{A_{n-1}}\) ? 
We start by  expressing in a precise way what we mean with `contains' \(\Cc_{A_{n-1}}\).

Recall that we write   \({\mathcal N}'(\Cc_{A_{n-1}})\) for the poset given
by the nested sets in \(\Cc_{A_{n-1}}\) that contain $\{V\}$, i.e. the poset that indicizes the strata of 
\(Y_{maxA_{n-1}}\).

\begin{proposition}
There is a graded  poset   embedding \(\varphi\) of  
\({\mathcal N}'(\Cc_{A_{n-1}})\)
into \({\mathcal N}'(\Bc(n-1) )\). 
\end{proposition}
\begin{proof}
Let \(\Tc\) be an element in 
\({\mathcal N}'(\Cc_{A_{n-1}})\). Then \(\Tc=\{ B_0=V, B_1,B_2,...,B_r\}\) is a nested set  of the building set \(\Cc_{A_{n-1}}\) containing \(V\). This means that its elements are linearly ordered by inclusion: \(V\supset B_1\supset \cdots \supset B_r\).
Now  we can express  every \(B_i\) as the direct sum of some irreducible subspaces \(A_{ij}\), i.e. elements of  \(\Fc_{A_{n-1}}\) (\(j=1,...,k_i\)).  We notice that, for every \(i=1,...,r\),  the sets  \(\Tc'_i=\{A_{sj}\} \cup \{V\}\) (with  \(s> r-i\) and, for every \(s\), \(j=1,...,k_s\))  is nested in \(\Fc_{{A_{n-1}}}\). The map \(\varphi\) defined by 
\[\varphi( \Tc)= \{\{V\}, \Tc'_1,...,\Tc'_r\}\]
if \(r\geq 1\), otherwise \[\varphi( \Tc)=\{\{V\}\}\]
is easily seen to be a poset embedding. 
\end{proof}

Given a  complex of nested sets   \(P\), we will  denote by  \(F^k(P)\) the subset made by the nested sets of cardinality \(k+1\). 

 The restriction of \(\varphi\)  to \(F^1({\mathcal N}'(\Cc_{A_{n-1}}))\) is an embedding of \(F^1({\mathcal N}'(\Cc_{A_{n-1}}))\) into \(F^1({\mathcal N}'(\Bc(n-1) ))\). Now \(F^1({\mathcal N}'(\Bc(n-1) ))\) can be identified with \(\Bc(n-1)\) (the identification maps \( \{\{V\}, \MS\}\), with \(\MS\) a nested set of \(\Fc_{{A_{n-1}}}\) that strictly contains  \(V\),  to \( \MS\)) and we still call \(\varphi\) the embedding  from 
$F^1( {\mathcal N}'(\Cc_{A_{n-1}}))$
 to \(\Bc(n-1)\). 
 More explicitly, if $B_1 \in \Cc_{{A_{n-1}}}$ is a subspace which is the direct sum of the irreducible subspaces $A_{11}, \ldots, A_{1k_1} $ then $$\varphi (\{V,B_1\})= \{V, A_{11}, \ldots, A_{1k_1}\}.$$

\begin{theorem}
\label{teo:chiusoazionebuilding}
The minimal building subset of  \(\Bc(n-1) \) which contains the image 
$\varphi(F^1({\mathcal N}'(\Cc_{A_{n-1}})))$
and is closed under the \(S_{n+1}\) action  is \(\Bc(n-1) \) itself.
\end{theorem}
\begin{proof}

Let us consider a building subset  \(\Gamma\) of  \(\Bc(n-1)\) that contains $\varphi(F^1({\mathcal N}'(\Cc_{A_{n-1}})))$  and is closed under the  \(S_{n+1}\) action. We will prove the claim by showing  that \(\Gamma=\Bc(n-1)\).
 

This can be done by induction on the depth of an element of \(\Bc(n-1)\),
which is defined in the following way: 
let \(\Tc\) be a \(\Fc_{A_{n-1}}\)-nested set that contains \(V\) and  consider the  levelled graph   associated to \(\Tc\). This graps is  an oriented tree: it coincides with the Hasse diagram of the poset induced by the  inclusion relation,  where  the leaves are the minimal subspaces of \(\MS\) and   the root
is  \(V\) and the orientation goes from the root to the leaves. 
A vertex $v$ is in level $k$ if the maximal length of a path that connects $v$ to a leaf is
$k$. 
We say that  \(\Tc\) has depth \(k\) if 
\(k\) is the  highest level of this tree. 
\footnote{We notice that this representation of a nested set by a tree is coherent with the one  introduced  in the Section \ref{secexamplean}: the only difference is that there we added  \(n\)  leaves.}

Now we prove by induction on \(k\) that every element in  \(\Bc(n-1)\) with depth \(k\) belongs to \(\Gamma\).

When \(k=0,1\) this is immediate: given \(B\neq V\in \Cc_{A_{{n-1}}} \), then \(\varphi(\{V, B\})\) is the nested set of depth 1 whose elements are \(V\) and the maximal elements of  \(\Fc_{{A_{n-1}}}\) contained in \(B\). In this way one can show that  all the elements of  \(\Bc(n-1)\)  with depth 1 that contain \(V\) belong to \(\Gamma\).

Let us check the case  \(k=2\).
One first observes that, in view of the definition of the  \(S_{n+1}\) action, every nested set of depth $2$ of the form \(\{V, B, B_1\}\), where \(B_1\subset B\), belongs to \(\Gamma\) since it
can be obtained as \(\sigma \MS\) for a suitable choice of \(\sigma\in S_{n+1}\) and of a the nested set of depth $1$ \(\MS\in \Bc(n-1)\). 
Now we show that also all  the nested sets of depth $2$ of the form \(\{V, B, B_1,...,B_j\}\), with $j \geq 2$ and \(B_i\subset B\) for every \(i\), belong to \(\Gamma\). In fact we can obtain \(\{V, B, B_1,...,B_j\}\) as a union, for every $i$, of the nested
sets \(\{V, B, B_i\}\) that belong to $\Gamma$ as remarked above.
Since all these sets have a nontrivial intersection $\{V,B\}$ and \(\Gamma \) is  building in the Feichtner-Kozlov sense (see Section \ref{sec:buildingcombinatorial}), this shows that \(\{V, B, B_1,...,B_j\}\) belongs to 
$\Gamma$.

Then let us consider a nested set of depth  $2$ \(\{V, B, B_1,...,B_j\}\) (\(j\geq 2 \) ), where
\begin{itemize}
\item[i)] $B$ is in level $1$;
\item[ii)] \(B_j\)  is  not  included in \(B\);
\item[iii)]$\,$ all the $B_i$'s are in level $0$.
\end{itemize}
This nested set is in \(\Gamma\)  since it can be obtained as a union  of the nested sets (with depth $1$) \(\MS_1=\{V, B_1,...,B_j\}\)  and \(\MS_2\), where \(\MS_2\) is any  nested subset of \(\{V, B, B_1,...,B_j\}\) with   depth $2$. We notice that \(\MS_1\) and \(\MS_2\)  are in \(\Gamma\), and have nonempty intersection, therefore their union belongs to  \(\Gamma \).

Now  we can show that in \(\Gamma\) there are all the  nested sets of depth $2$: if the set \(\{V, C_1,...C_s, B_1,...,B_j\}\) has depth $2$, where \(s\geq 2\) and the subspaces      \(C_1,...,C_s\)   are in level $1$, while the \(B_i\)'s are in level $0$,
we can obtain \(\{V, C_1,...C_s, B_1,...,B_j\}\) as the union of the nested sets \(\{V, C_i, B_1,...,B_j\}\) (for every \(i=1,..,s\)) that   have pairwise nonempty intersection and belong to \(\Gamma\), as we have already shown. 

Let us now consider   \(k\geq 2\) and  suppose that every nested set in \(\Bc(n-1)\) with depth  \(\leq k\) belongs to \(\Gamma\).  Let \(\Tc\) be a nested set of depth \(k+1\).
 Let us denote by \(\Tc_k\) the nested set obtained removing from \(\Tc\) the 
subspaces in level $0$:
it belongs to \(\Gamma\) by the inductive hypothesis. Then we consider the nested set \(\Tc'\) obtained removing from \(\Tc\) the levels \(2,..., k\): since \(\Tc'\)  has depth 2  it belongs to \(\Gamma\) again by the inductive hypothesis.
We observe that  \(\Tc_k\) and \(\Tc'\) have nonempty intersection, therefore their union \(\Tc\) belongs to \(\Gamma\).
 \end{proof}

%

\section{Supermaximal models and cohomology}
\subsection{The Model  \(Y_{\Bc(n-1)}\) is a Supermaximal Model}
\label{subsec:supermaximalsvelato}

We can     now answer  to the question,  raised in  Section \ref{sec:supermaximal}, about how to construct a model  that is `bigger' than the maximal model, admits the extended \(S_{n+1}\) action and is minimal with these properties. 

Let us state this in a more formal way. 
Let us consider  the poset  \(\Bc(n-1)\) that    indicizes  the  strata of the minimal model \(Y_{minA_{n-1}}\), and let us denote  by \({\mathcal L}\) the family of the 
models  obtained by blowing up all the  building subsets of these strata. We observe that \({\mathcal L}\) has a natural poset structure given by the relation
\(Y_{\G_1}\leq Y_{\G_2}\) if and only if   \(\G_1\subseteq \G_2\) (by Li's definition, this also means that there is a birational projection of \(Y_{\G_2}\) onto \(Y_{\G_1}\)). 

 Let us denote by \(\Tc\) the set the elements of \(\Bc(n-1)\) with  depth 1. 
From Theorem \ref{teo:listabuilding} and Remark \ref{rem:ordinescoppiamenti} it follows that if  we blowup in \(Y_{minA_{n-1}}\)  the strata  that correspond to the elements of \(\Tc\) (in a suitable order, i.e. first the strata with smaller dimension) we obtain the model \(Y_{maxA_{n-1}}\). 

\begin{definition}
The supermaximal model $Y_{supermaxA_{n-1}}$ associated with the root arrangement \(A_{n-1}\) is the minimal model \(Y_\Kc\) in the poset \({\mathcal L}\) that admits the \(S_{n+1}\) action and such that \(\Kc\supseteq \Tc\). 
\end{definition}

We notice that this last property means that the supermaximal $Y_{supermaxA_{n-1}}$ model admits a birational projection onto \(Y_{maxA_{n-1}}\).

As a consequence of Theorem \ref{teo:chiusoazionebuilding} we have proven the following result:
\begin{theorem}
The model \(Y_{\Bc(n-1)}\) is the supermaximal model associated with the root arrangement \(A_{n-1}\). 
\end{theorem}

\begin{remark}
There is a family of \(S_n\)-invariant building sets that are intermediate between $\Fc_{A_{n-1}}$ and
$\Cc_{A_{n-1}}$
(these building sets have been classified in \cite{GaiffiServenti2}). 
Let $\Ec$
be such a building set and let \(Y_\Ec\) be the corresponding  model. 
We will denote by \(Y_{super\Ec}\) the minimal model in  \({\mathcal L}\) among the models \(Y_{\Kc}\) that  admit the \(S_{n+1}\)-action and such that \(\Kc\supseteq \Ec\). 
Depending on the choice of $\Ec$, it may be $Y_{super\Ec} \lneqq Y_{\Bc(n-1)}$. This happens for
instance when $\Ec$ is the  building set that contains $\Fc_{n-1}$
and all the triples $\{ V, A_1, A_2\} \in \Bc(n-1)$ such that the sum of $A_1$ and $A_2$ is direct and
has dimension $n-2$ (this building set is denoted by \(\G_2(A_{n-1})\) in  \cite{GaiffiServenti2}).
 \end{remark}

\subsection{The Cohomology of a Complex Supermaximal Model}
The discussion in the preceding sections  points out the interest of the supermaximal models and of the corresponding symmetric group actions. 

 Let \(\Fc\) be the  building set of irreducible subspaces associated with a subspace arrangement  in a complex vector space \(V\) of dimension \(n\).
 In Section \ref{sec:buildingcombinatorial} we defined the building set 
$\Bc(n-1) = \Nc'(\Fc_{A_{n-1}})$, the building set for the supermaximal model for $A_{n-1}$. By  analogy with this notation we write $\Bc(\Fc)$ for the building set $\Nc'(\Fc)$, that
generalizes in a way the idea of supermaximal model.

Let \(Y_{\Bc(\Fc)}\) be the model obtained by blowing up all the strata of the minimal model \(Y_{\Fc}\). 
We recall from Example \ref{es:boundary} that the 
strata of \(Y_{\Bc(\Fc)}\) are in bijection with the nested sets in $\Bc(\Fc)$, according to the constructions given in \cite{procesimacpherson, li}.

%
%

Let us denote by \(\upi^{\Bc(\Fc)}_\Fc\) the projection from \(Y_{\Bc(\Fc)}\) onto \(Y_\Fc\).

\begin{theorem}\label{thm:cohomology}
A basis of the  integer cohomology of the complex model \(Y_{\Bc(\Fc)}\) is given by the following monomials:
\[  \eta\ c_{\Sc_1}^{\delta_1}c_{\Sc_2}^{\delta_2}\cdots c_{\Sc_k}^{\delta_k}  \]
where 
\begin{enumerate}
\item  \(\Sc_1\subsetneq \Sc_2 \subsetneq\cdots \subsetneq \Sc_k\) is a chain of  \(\Fc\)-nested sets (possibly empty, i.e. \(k=0\)),  with \(\{V\}\subsetneq \Sc_1\);
\item  the exponents \(\delta_i\), for \(i=1,\ldots, k\),  satisfy the following inequalities: \(1\leq \delta_i \leq |\Sc_i|-|\Sc_{i-1}|-1\), where we put \(\Sc_0=\{V\}\);
\item \(\eta \) belongs to \((\upi^{\Bc(\Fc)}_\Fc)^*H^*(\Dc_{\Sc_1})\) (if \(k\geq 1\)) or to \((\upi^{\Bc(\Fc)}_\Fc)^*H^*(Y_\Fc)\) (if \(k=0\))
and is the image, via $(\upi^{\Bc(\Fc)}_\Fc)^*$, of a monomial in the  Yuzvinsky basis (see Section \ref{constru});
\item the element $c_{\Sc_i}$ is the Chern class of the normal bundle of $L_{\Sc_i}$ (the proper transform of  $\Dc_{\Sc_i}$) in \(Y_{\Bc(\Fc)}\).
\end{enumerate}
\end{theorem}

\begin{proof}
Let us fix some notation.
We construct \(Y_{\Bc(\Fc)}\) starting from \(Y_\Fc\) by choosing a sequence of blowups (we start by blowing up the 0-dimensional strata, then the 1-dimensional strata, and so on, see Remark \ref{rem:ordinescoppiamenti}). At a certain step of this blowup process we have blown up a  set of strata indicized by a subset \(\A\) of \(\Bc(\Fc)\). We then denote by \(Y_\A\) the variety we have obtained. 
  This means that \(\A\) has  the following property: there exists an integer \(k\), with  \(2\leq k\leq n+1\) such that \(\A\)  contains all the nested sets with cardinality \(\geq k+1\) and doesn't contain any nested set with cardinality \(<k\).
\begin{remark}
According to this notation, the variety \(Y_\Fc\) can also be denoted by \(Y_{\emptyset}\), i.e., it is the variety we have when no stratum has been blown up.
\end{remark}

As a further notation, if \(Y_{\A_1}\) and \(Y_{\A_2}\) are  two varieties  obtained during the blowup process with \(\A_1\subset \A_2\subset \Bc(\Fc)\),  we denote by \(\pi^{\A_2}_{\A_1}\: : \:  Y_{\A_2}\rightarrow Y_{\A_1} \) the blow up map. Let us recall the following lemma.

\begin{lemma}[Keel, \cite{keel}] \label{lem:keel}
Given \(\A\) as before, let us suppose that the next stratum that we have to blow up is indicized by \(\Sc\in \Bc(\Fc)\setminus  \A\)  with \(|\Sc|=k\). We call $\Dc_\Sc'$ the proper transform of $\Dc_\Sc$ in $Y_\A$ and
$\widetilde{\Dc_\Sc}$ the proper transform of $\Dc_\Sc$ in $Y_{\A \cup \{\Sc\}}.$ Then
\[H^*(Y_{\A\cup \{\Sc\}})\cong  
H^*(Y_{\A})[\zeta] /(J \cdot \zeta, P_{Y_\A/\Dc_\Sc'}(-\zeta)) 
\]
where $\zeta$ is the class of the proper transform
\(\widetilde{\Dc_\Sc}\) 
of \(\Dc_\Sc\) in $H^2(Y_{\A\cup \{\Sc\}}),$
$J$ is the kernel of the restriction map $H^*(Y_\A) \to H^*(\Dc_\Sc')$ and
$P_{Y_\A/\Dc_\Sc'}[x]$ is any polynomial in $H^*(Y_\A)[x]$ whose restriction to $H^*(\Dc_\Sc')[x]$ is the Chern polynomial of the
normal bundle of 
$\Dc_\Sc'$ in $Y_\A.$
\end{lemma}

Using the lemma we can prove the claim of the theorem by induction. We fix a sequence of blowups 
$$X_0 \ot^\pi X_1 \ot^\pi \cdots \ot X_N$$ that constructs \(Y_{\Bc(\Fc)}=X_N\) starting from \(Y_\Fc = X_0\). 
We assume by inductive hypothesis that a basis of the integer cohomology of $X_n$ is given by the monomials 
\[  \eta\ c_{\Sc_1}^{\delta_1}c_{\Sc_2}^{\delta_2}\cdots c_{\Sc_k}^{\delta_k}  \]
where 
\begin{enumerate}
  \item \(\Sc_1\subsetneq \Sc_2 \subsetneq\cdots \subsetneq \Sc_k\) is a possibly empty chain of  \(\Fc\)-nested sets already blown up;
  \item \(\eta \) belongs to \((\pi^{X_n}_{X_0})^*\iota^*H^*(\Dc_{\Sc_1})\) if \(k\geq 1\) ($\iota$ is the inclusion map of $\Dc_{\Sc_i}$ in $X_0$), and to \((\pi^{X_n}_{X_0})^*H^*(X_0)\) if \(k=0\) and is the image of a monomial in the  Yuzvinsky basis;
  \item the term $c_{\Sc_i}$ is the Chern class of the normal bundle of $\Dc'_{\Sc_i}$, that is the proper transform of $\Dc_\Sc$ in $X_n;$
  \item the exponents $\delta_i$ satisfy the inequalities \(1\leq \delta_i \leq |\Sc_i|-|\Sc_{i-1}|-1\), where we put \(\Sc_0=\{V\}.\)
\end{enumerate}

Let $\Dc_\Sc'$ be the stratum in $X_n$ that we have to blowup in order to get $X_{n+1}.$
We assume that  ${\Dc_\Sc'}$ is the proper transform in $X_n$ of the stratum $\Dc_\Sc$ in $X_0,$
corresponding to the $\Fc$-nested set $\Sc \in \Bc(\Fc).$ Let $d= | \Sc |.$
Since at the step $X_n$ we have already blown up all the strata $\Dc_\Tc$ with $| \Tc |> d,$ the submanifold
$\Dc_\Sc'$ is a complex supermaximal model of smaller dimension. 
In fact, since we have that, as in Definition \ref{def:simple}, 
$$T(\Dc_\Sc \cap \Dc_\Tc) = T(\Dc_\Sc)_{\Dc_\Sc \cap \Dc_\Tc} \cap T(\Dc_\Tc)_{\Dc_\Sc \cap \Dc_\Tc},$$ 
the stratum  
$\Dc_\Sc'$ can be obtained as a blowup of
$\Dc_\Sc$ along the strata $\Dc_\Tc \cap \Dc_\Sc$ with $| \Tc |> d.$ 
Hence we can assume by induction on the dimension of $\Dc_\Sc'$, that we know the cohomology $H^*(\Dc'_\Sc),$ according to the statement of the theorem.

Then, by Lemma \ref{lem:keel} we have that the cohomology of $X_{n+1}$ is given by

\[H^*(X_{n+1})\cong  
H^*(X_n)[\zeta] /(J \cdot \zeta, P_{X_n/\Dc_\Sc'}(-\zeta)) 
\]
where $J$ is the kernel of the projection induced by the inclusion of $\Dc_\Sc'$ in $X_n$
and $\zeta = c_{\Sc}$ is the Chern class of the normal bundle of the proper transform $\widetilde{\Dc_\Sc}$ of $\Dc_\Sc$ in $X_{n+1}.$ 
The polynomial $P_{X_n/\Dc_\Sc'}(-\zeta)$ is the Chern polynomial of the normal bundle of $\Dc_\Sc'$ in $X_n,$ that has rank $d =|\Sc|$
and hence $\deg P_{X_n/\Dc_\Sc'}(-\zeta) = d - 1.$ 

It follows that a basis of $H^*(X_{n+1})$ is given by the union of two set of generators:
\begin{enumerate}
  \item the monomials
\[ \mu_1= \eta\ c_{\Sc_1}^{\delta_1}c_{\Sc_2}^{\delta_2}\cdots c_{\Sc_k}^{\delta_k}  \]
that are already in the base of $H^*(X_n)$ and that we can identify with the
corresponding generators of the cohomology of $X_{n+1}$ via the pull back
$\pi^*: H^*(X_n) \to H^*(X_{n+1});$ 
\item the monomials
\[ \mu_2 = \eta\ c_\Sc^{\delta} c_{\Sc_1}^{\delta_1}c_{\Sc_2}^{\delta_2}\cdots c_{\Sc_k}^{\delta_k} \]
where $0 < \delta < |\Sc|-1$ and $\eta\ c_{\Sc_1}^{\delta_1}c_{\Sc_2}^{\delta_2}\cdots c_{\Sc_k}^{\delta_k}$ is a monomial among the generators of the cohomology 
$H^*(\Dc_\Sc');$ we identify $\mu_2$ with a generator of the cohomology of $H^*(X_{n+1})$ as follows:
given the monomial $\eta\ c_{\Sc_1}^{\delta_1}c_{\Sc_2}^{\delta_2}\cdots c_{\Sc_k}^{\delta_k} \in H^*(\Dc_\Sc'),$
we identify it with the coset $\eta\ c_{\Sc_1}^{\delta_1}c_{\Sc_2}^{\delta_2}\cdots c_{\Sc_k}^{\delta_k}+J \in H^*(X_n)$ by the projection induced by the inclusion of $\Dc_\Sc'$ in $X_n;$ hence when we multiply it by $c_\Sc,$ that is the Chern class of the normal bundle of the divisor $\widetilde{\Dc_\Sc} \in X_{n+1},$ this gives a well defined class \[\eta\ c_\Sc^{\delta} c_{\Sc_1}^{\delta_1}c_{\Sc_2}^{\delta_2}\cdots c_{\Sc_k}^{\delta_k} \in H^*(X_{n+1}).\]  
In particular, since we are considering the class
$c_{\Sc_1}$ that is the Chern class of the normal bundle of the proper transform $\Dc_{\Sc_1}'$ of $\Dc_{\Sc_1}$ in $\Dc_\Sc'$, the exponent 
$\delta_1$ will be at most $|\Sc_1| - |\Sc| -1,$ that is the dimension of the projectivized normal bundle
of $\Dc_{\Sc_1}'$ in $\Dc_\Sc'.$ 
\end{enumerate}

Let \(\Sc, \Tc\), with $\Sc \subset \Tc$ be two $\Fc$-nested sets.   Let $\Dc_\Sc'$ (resp. $\Dc_\Tc'$) be the proper transform of $\Dc_\Sc$ (resp. $\Dc_\Tc$) in $X_n.$ 
We assume 
that in $X_n$ we have already performed the blowup of the stratum associated to $\Dc_\Tc.$
Since we have the inclusion $\Dc_\Tc \subset \Dc_\Sc,$ we can also consider the proper transform $\Dc_\Tc' \cap \Dc_\Sc'$
of $\Dc_\Tc$ in $\Dc_\Sc'$.
 
Let $\zeta_\Tc$ be the Chern class of the normal bundle of $\Dc_\Tc'$ in $X_n$ and let $\zeta_\Tc^\Sc$ be
the Chern class of the normal bundle of $\Dc_\Tc'\cap \Dc_\Sc'$ in $\Dc_\Sc'.$
We claim that the projection $i^*$ induced by the inclusion $i:\Dc_\Sc' \subset X_n$ maps
$i^*:\zeta_\Tc \mapsto \zeta_\Tc^\Sc.$
This follows since the Thom class $\tau_\Tc$ of the normal bundle of $\Dc_\Tc'$ in $X_n$ restricts to the Thom class $\tau_\Tc^S$ of the normal bundle of $\Dc_\Tc'\cap \Dc_\Sc'$ in $\Dc_\Sc$ and we have that the Chern class of a line bundle is the pullback of the zero-section of its Thom class.

Finally we notice that, since Chern classes are functorial, whenever we perform a blowup $\pi:X_{n+1} \to X_n$, the Chern class $c_\Tc$ of a divisor $\Dc_\Tc'$ in $X_n$ pulls back to the Chern class of its proper transform $\widetilde{\Dc_\Tc}$  and hence we can identify them in our notation. 

\end{proof}

%

\section{Poincar\'e series and Euler characteristic of supermaximal models
}
\label{sec:poincaresupermodels}
In this section we use  the cohomology basis described in Theorem \ref{thm:cohomology} to prove   a formula for a series that encodes all the information regarding the Poincar\'e polynomials of the supermaximal models \(Y_{\Bc(n-1)}\).

We start by recalling the analogue computation in the case of minimal models.
The  Poincar\'e series \[\Phi(q,t)=t+ \sum_{n\geq 2}\sum_i dim \ H^{2i}(Y_{minA_{n-1}}, \Z) q^i \frac{t^n}{n!}\] for the minimal De Concini-Procesi models \(Y_{minA_{n-1}}\)  has been computed in many different ways.  
One can see for instance  \cite{Manin}, \cite{YuzBasi}, \cite{GaiffiBlowups},  \cite{getzler};  a formula for another series that encodes the same information   is provided   in Section \ref{ss:action3} of the present paper. 
 
Both in   Section 5 of \cite{YuzBasi} and  Section 4 of \cite{GaiffiBlowups} the computation of \(\Phi(q,t)\) consists in counting  the elements of  the Yuzvinski basis for \(H^{*}(Y_{minA_{n-1}}, \Z)\) described in Section \ref{constru}. 
We recall that, given  a monomial \(m_f\) in this basis, the set \(supp \ f\) is a nested set and that in  Section \ref{secexamplean} we represented the $\Fc_{A_{n-1}}$- nested sets, i.e. the elements of \(\Nc( \Fc_{A_{n-1}}), \) as oriented forests on \(n\) leaves. Let us then denote by \(\lambda(q,t)\)  the contribution  to \(\Phi(q,t)\)  provided by the basis monomials whose associated nested set is represented by a tree.

It turns out that \(\lambda(q,t)\)  satisfies the following recursive relation: 
\begin{displaymath}
	\lambda(q,t) ^{(1)}=1 + \frac{\lambda(q,t) ^{(1)}}{q-1}
	\left[e^{q\lambda(q,t) }-qe^{\lambda(q,t) }+q-1    \right]
\end{displaymath}
(here the superscript \(^{(1)}\) means the first derivative with respect to \(t\)).
Then one obtains a formula for   \(\Phi(q,t)\) by observing that   \(\Phi(q,t)= e^{\lambda(q,t) }-1\).


Let us now denote by \(\Phi_{super}(q,t)\) the  Poincar\'e series: 
\[\Phi_{super}(q,t)=t+ \sum_{n\geq 2} \sum_i dim \ H^{2i}(Y_{\Bc(n-1)}, \Z) q^i \frac{t^n}{n!}\]
\begin{definition}
We define the following series in four variables: 
\begin{equation}
\label{formulacsi}
 \xi(t,q,y,z)=\Phi(q,t) + \sum_{
\begin{array}{c}
 n \geq 2  \\
\MS\in \Nc'( \Fc_{A_{n-1}}) \\
 |\MS|=\ell +1 > 1   
\end{array}}\sum_{r\geq 0, k\geq 0}dim \ H^k(\Dc_\MS)N_{r,\MS} \ y^\ell z^r q^k\frac{t^n}{n!}
\end{equation}
where   \(\Dc_\MS\) is the subvariety in the boundary of  \(Y_{min A_{n-1}}\) defined in Section \ref{constru}  
%
and  \(N_{r,\MS}\) is the number of nested sets in \(\Nc'( \Fc_{A_{n-1}})\)   that contain \(\MS\) and whose cardinality is \(|\MS|+ r\).

\end{definition}
\begin{theorem}
\label{teo:sostituzioni}
One can obtain  the Poincar\'e series \(\Phi_{super}(q,t)\)  from the series  \(\xi(t,q,y,z)\)  by substituting:
\begin{itemize}

\item  \(y^\ell\)  with \(\frac{q^\ell-q}{q-1}\); 
\item  \(z^r\)   with \[\sum_{s\leq r, \;  j_0=0< j_1< \cdots  < j_s=r} \frac{r!}{j_1!(j_2-j_1)!\cdots (j_s-j_{s-1})!}\prod_{\theta=1}^{s}q\frac{q^{j_{\theta}-j_{\theta-1}-1}-1}{q-1}\]

\end{itemize}
\end{theorem}
\begin{proof}
The monomials in the basis of   \(H^*(Y_{\Bc(n-1)})\) (\(n\geq 2\)) are described by  Theorem \ref{thm:cohomology}: 
\[  \eta\ c_{\MS_1}^{\delta_1}c_{\MS_2}^{\delta_2}\cdots c_{\MS_k}^{\delta_k}  \]
where \(\eta\) is  represented by a monomial in the basis of  \( H^*(\Dc_{\MS_1})\) if \(k\geq 1\) and of \( H^*(Y_{min A_{n-1}})\) if \(k=0\).
The monomials with \(k=0\) are computed by the series \(\Phi(q,t)\) that is the first addendum in the formula (\ref{formulacsi}). 
Then we observe that, once \(\MS_1\) is fixed, if  the  exponent of the variable \(z\) is \(r\) this means that we are keeping into account all the monomials such that \(|\MS_k-\MS_1|=r\). This is expressed by the substitution formula  for \(z^r\). 

The exponent of the variable \(y\) coincides with \(|\MS_1-\{V\}|\) and  the exponent  \(\delta_1\) satisfies \(1\leq \delta_1\leq |S_1-\{V\}|-1\): this is expressed by the substitution of \(y^\ell\)  with \(\frac{q^\ell-q}{q-1}\) (notice that \(\ell>1\)).

\end{proof}

We are therefore interested in  finding a formula for \(\xi(t,q,y,z)\).
Let us denote by  \(p_n(q)\), for \(n\geq 2\),  the Poincar\'e polynomial of the minimal model \(Y_{minA_{n-1}}\), i.e. we can write    \(\Phi(q,t)=t+ \sum_{n\geq 2}p_n(q)\frac{t^n}{n!}\).
Then we consider the series  \(W(t,z)=\sum_{n\geq 2}w_n(z) t^n\)  where the polynomials \(w_n(z)\) count the number of elements of $\Bc(n-1)$. More precisely  \[w_n(z)=\sum_{1\leq j\leq n-1}\sum_{
\begin{array}{c}
  \MS\in \Bc(n-1)    \\
  |\MS|=j    
\end{array}
}z^{j-1}\]
As we will  recall in Section \ref{sec:combinatorialaction},
we have 
\[w_n(z)=\sum_{1\leq j\leq n-1}|\Pc_2(n+j-1,j)|z^{j-1}\]
where we  denote by \(\Pc_2(n+j-1,j)\)  the number of unordered partitions of the set \(\{1,..., n+j-1\}\) into \(j\) parts of  cardinality greater than or equal to 2 (these numbers are  the 2-associated Stirling numbers of the second kind, see  for instance the table at  page 222 of \cite{comtet}). 

If we now put  \[\psi(t,q,z)=\sum_{n\geq 2}p_n(q)w_n(z)\frac{t^n}{n!}      \]
we can write a formula that computes  the series \(\xi(t,q,y,z)\):
\begin{theorem}
\label{teo:formulapoincare}
We have 
\[  \xi(t,q,y,z)= \Phi(q,t)+ \sum_{\nu\geq 1}  \psi(t,q,z)^{(\nu)}  \frac{\Gamma^\nu}{\nu !} \]
where 
\[\Gamma=\sum_{\ell\geq 1}\frac{1}{\ell !}y^\ell (\psi(t,q,z)^\ell)^{(\ell-1)}\]
and  the superscript \(^{(j)}\) means the \(j\)-th derivative with respect to \(t\).
\end{theorem}
The next (sub)section will be devoted to the proof of this theorem.
\begin{example}
Here it is a computation of the first terms of the series  \(\xi(t,q,y,z)\), that is the terms whose \(t\)-degree is \(\leq 5\): 
%
$${\frac{t^5}{5!}}\left[\left(315\,q^2+1305\,q+315\right)yz^2+\left(\left(
 315\,q+315\right)y^2+\left(210\,q^2+870\,q+210\right)y\right)z
 \,+\right.
$$
$$
\left.+\,105\,y^3+\left(105\,q+105\right)\,y^2+\left(25\,q^2+95\,q+25\right)
 \,y+q^3+16\,q^2+16\,q+1\right]\,+
$$
$$
+\frac{t^4}{4!}\left[30yz(1+q)+15 y^2+10y(1+q) +q^2+5q+1\right] +\frac{t^3}{3!}\left[3y+q+1\right] +\frac{t^2}{2}+t.$$
\end{example}

If we put \(q=1\) in the formula for the Poicar\'e series \(\Phi_{super}(q,t)\), we obtain the Euler characteristic series of  the supermaximal models \(Y_{\Bc(n-1)}\). Moreover we notice that if we put \(q=-1\) we obtain the Euler characteristic series of the {\em real} supermaximal models \(Y_{\Bc(n-1)}(\R)\) constructed with base field \(\R\). 
In fact from a result of  \cite{krasnov} it follows   that
\(H^{2i}(Y_{\Bc(n-1)}, \Z_2)\cong H^{i}(Y_{\Bc(n-1)}(\R), \Z_2)\);  therefore \(\sum_{i}(-1)^i \dim\ H^{2i}(Y_{\Bc(n-1)}, \Q)\)
 is equal to the Euler characteristic  \(\chi_E (Y_{\Bc(n-1)}(\R))\). 
 
We observe that one can obtain  the Euler characteristic series in a more direct way  from the series \(\xi(t,q,y,z)\):

\begin{theorem}
\label{teo:eulerchar}
One can compute the Euler characteristic series of  the models \(Y_{\Bc(n-1)}(\R)\) from  the series \(\xi(t,q,y,z)\) 
by  substituting  \(q\) with \(-1\) and also
\begin{itemize}

\item \(y^k\) with \(-1\) if \(k\) is even, otherwise with 0; 

\item \(z^r\)  with \(E_r\), where we denote by \(E_n\)  the Euler secant number defined by
\[\frac{2}{e^t+e^{-t}}=\sum_nE_n\frac{t^n}{n!}\] (see sequence A028296 in OEIS; notice that if \(n\) is odd then \(E_n=0\)).

\end{itemize}
\end{theorem}
\begin{proof}

When we find \(z^r\) in \(\xi(t,q,y,z)\), it means that we are computing the contribution to the Euler characteristic of all the  factors  
\( c_{\MS_1}^{\delta_1}c_{\MS_2}^{\delta_2}\cdots c_{\MS_k}^{\delta_k}\) of the basis  elements,   such that \(|\MS_k-\MS_1|=r\).
We observe that this contribution is non zero if and only if \(r\) is even. In this case our problem is equivalent to  computing the Euler characteristic of the order  complex   of the poset \(\Pc(\{1,...,r\})_{even}\) of the subsets of \(\{1,...,r\}\) with even cardinality.
By  Philip Hall Theorem (see for instance Proposition 1.2.6  of \cite{wachsposet})  we can compute it via the Moebius function, that  is equal to \(E_r\)  (see  Section 3.7 of Stanley's  survey   \cite{stanleysurvey}).

\end{proof}

\subsection{Proof of Theorem \ref{teo:formulapoincare}}
We will follow a strategy similar to the one in Section 3 of \cite{Gaiffigeneralizedpoincare}.

 \begin{definition}
Given any rooted oriented  tree $T$, the polynomial  ${\cal Q}( T)$ 
is a product of 
monomials that contains a factor for each vertex  of $T$: if $v$ 
is a vertex of $T$ with  $\nu $ outgoing edges, the 
corresponding factor is $y^\nu{\displaystyle \psi(t,q,z)^{(\nu )}}$. 
\end{definition}

We notice that  the cardinality  
${\displaystyle Aut (T)}$  of the automorphism group of a rooted 
tree  $T$  is equal to  the product ${\displaystyle \prod 
_{v}^{}\gamma _{sym,v}}$ where $v$ ranges over the vertices of $T$ and $\gamma _{sym,v}$ is determined in this way:  delete $v$ and consider the 
connected components of the subgraph of $T$ that stems from 
$v$. Suppose that they can be partitioned  in $k$ automorphism classes with 
the following cardinalities: $a_1,a_2,\ldots,a_k$. Then $\gamma 
_{sym,v}=a_1!a_2!\cdots a_k!$. 

Now  we  associate a tree  to every monomial \( \eta\ c_{\MS_1}^{\delta_1}c_{\MS_2}^{\delta_2}\cdots c_{\MS_k}^{\delta_k}\) in the basis of \(H^*(Y_{\Bc(n-1)})\),  by choosing the rooted tree  with \(n\) leaves associated with the nested set \(\MS_1\) (see Section 
\ref{secexamplean}) and deleting the leaves and the edges that contain a leaf.
Then \(\xi(t,q,y,z)- \Phi(q,t)\) can be computed by regrouping together all the basis monomials that are associated with the same rooted tree. This means that \(\xi(t,q,y,z)- \Phi(q,t)\) can be written  as:
  
  
\begin{displaymath}
	\sum _{n\geq 1}^{}\psi(t,q,z)^{(n)}\frac{(y\Gamma)^n}{n!}
\end{displaymath}
where
\begin{displaymath}
	\Gamma=\sum 
	_{[T] }
	\frac{{\cal Q}( T)}{Aut (T)}
\end{displaymath}
and the sum ranges over all the automorphism classes of  nonempty oriented, 
rooted trees.

\noindent Thus the problem can be reduced to the one of finding a `nice' formula for 
$\Gamma$. This is provided by the following theorem, that is a consequence of Theorem 3.3   of \cite{Gaiffigeneralizedpoincare}.

\begin{theorem}
\begin{equation}
	\sum 
	_{
	\begin{array}{c}
	[T]  	
	\end{array}
	}^{}\frac{{\cal Q}( T)}{Aut (T)}=
	\sum _{n\geq 1}^{}\frac{1}{n!}y^{n-1}\left(\psi(t,q,z) 
	^n \right)^{(n-1)}
	\label{bellaeq}
\end{equation}
where $[T]$ ranges over all the automorphism classes of 
nonempty oriented  rooted trees.
\label{belte}
\end{theorem}

\section{A combinatorial extended \(S_{n+k}\) action on the poset of the boundary strata of   $Y_{\Fc_{A_{n-1}}}$} \label{sec:combinatorialaction}
This section and the next one are devoted to point out that another hidden  extended action of the symmetric group appears in the geometry of   the minimal models  $Y_{\Fc_{A_{n-1}}}$. This action   is different from the  one,  described in Section \ref{sec:action1},  that motivated the construction of supemaximal models.
More precisely, we are going to deal with a purely combinatorial action  on the poset  \(\Bc(n-1)\) that indexes  the strata of $Y_{\Fc_{A_{n-1}}}$.
This does not correspond to an action on the variety $Y_{\Fc_{A_{n-1}}}$, but it gives rise, as we will see, to an interesting permutation action on the monomials of the Yuzvinski basis of  $H^*(Y_{\Fc_{A_{n-1}}}, \Z)$.

Let us denote by \(F^k(\Bc(n-1))\) the subset of \(\Bc(n-1)\) made by the elements of cardinality \(k+1\).
These elements indicize the \(k\)-codimensional strata of $Y_{\Fc_{A_{n-1}}}$ (the only element in \(F^0(\Bc(n-1))\) is \(\{V\}\) that  corresponds to the big open part).
In \cite{gaifficayley} it has been described an explicit bijection between \(F^k(\Bc(n-1))\) and the set of unordered partitions of \(\{1,2,...,n+k\}\) into \(k+1\) parts of cardinality greater than or equal to 2.
To recall this bijection, we identify the elements of \(\Fc_{A_{n-1}}\) with subsets of \(\{1,2,...n\}\), as in Section \ref{secexamplean}.

\begin{definition}
\label{ordering}
We fix the following  (strict) partial ordering on  \(\Fc_{A_{n-1}}\): given  \(I\) and \(J\) in \(\Fc_{A_{n-1}}\)  we put \(I< J\)  if the minimal number  in \(I\) is less than the minimal number  in \(J\).
\end{definition}

Let us  consider a nested set   \(\MS\)   that belongs to \(F^k(\Bc(n-1))\). 
 It can be represented by an oriented  rooted tree on \(n\) leaves as in Section \ref{secexamplean}. The leaves are the sets \(\{1\},\{2\},\cdots,\{n\}\).
 Now we put labels on the vertices of this tree. We start by labelling  the vertices \(\{1\},\{2\},\cdots,\{n\}\) respectively by the labels \(1,2,...,n\). 

Then we can partition the set of vertices of the tree into  levels
with the same criterion as in the proof of Theorem
\ref{teo:chiusoazionebuilding}:  level $0$ is made by the leaves, and in general,   level \(j\)  is made by the vertices \(v\) such that the maximal length of an oriented  path that connects \(v\) to a leaf is \(j\). 

Now  we  label the internal vertices of the tree in the following way. Let us suppose that there are \(q\) vertices in level 1. These  vertices correspond, by the nested property, to pairwise disjoint  elements of  \(\Fc_{A_{n-1}}\), therefore   we can totally order them using the ordering of Definition \ref{ordering} and  we  label them with  the numbers from \(n+1\) to \(n+q\) (the label  \(n+1\)  goes to the minimum, while \(n+q\) goes to the maximum).

At the same way, if there are   \(t\) vertices in level 2, we can label them with the numbers from \(n+q+1\) to \(n+q+t\) , and so on. At the end of the process, the root is labelled with the number \(n+k+1\).

We can now  associate to such a tree an unordered  partition of \(\{1,2,...,n+k\}\) into \(k+1\) parts  by assigning  to every internal vertex \(v\) the set of the labels of the vertices covered by \(v\) (see Figure \ref{labellednestedtree}).

 \begin{figure}[ht]

 \center
\includegraphics[scale=0.5]{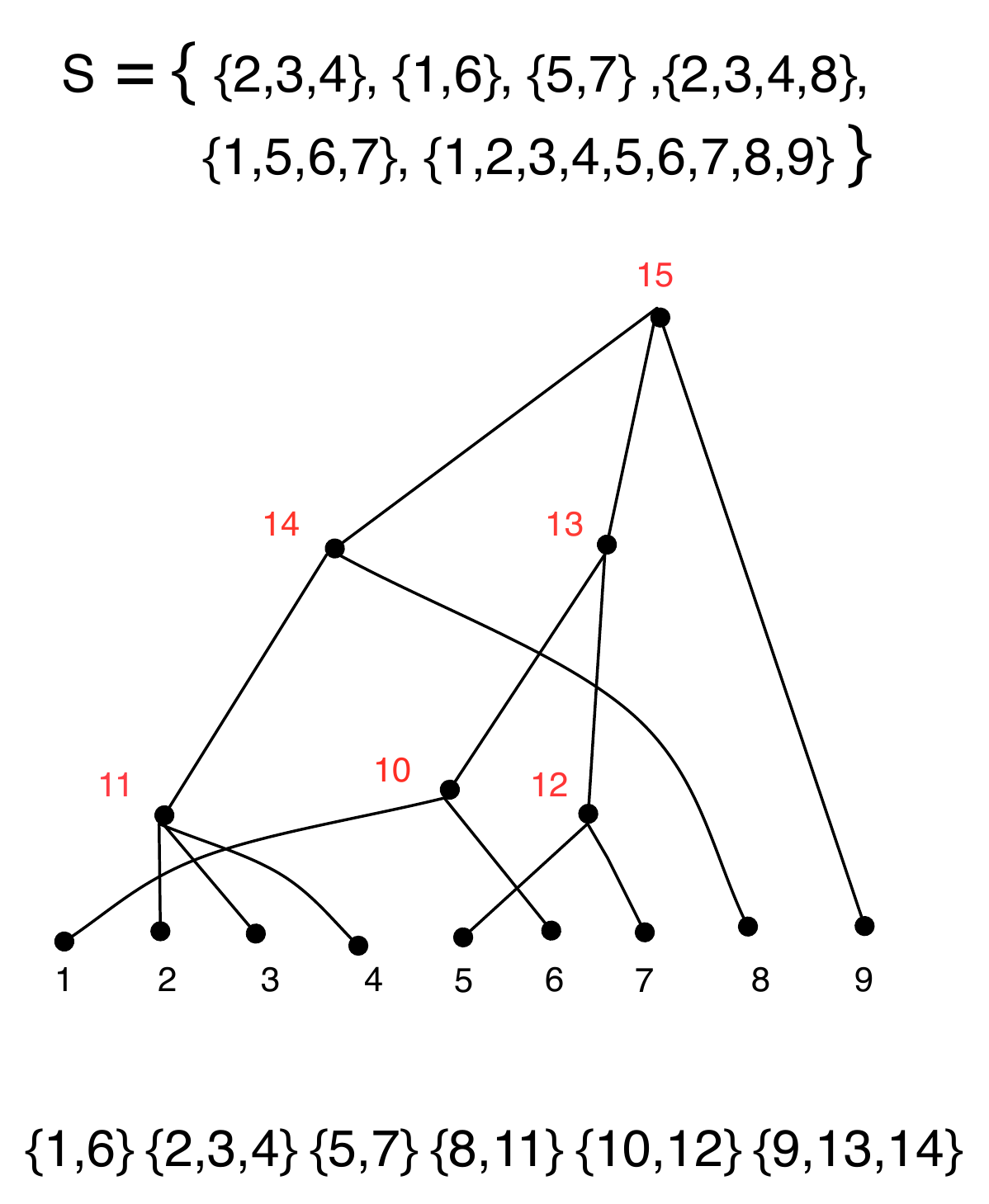}
\caption{On top of the picture there is a  nested set \(S\) with 6 elements  in \(\Fc_{A_{8}}\).  In the middle there is  its representation by an oriented labelled rooted tree. At the bottom one can read  the resulting partition  of \(\{1,2,...,14\}\) into 6 parts.}
 \label{labellednestedtree}
\end{figure}

This bijection allows us to consider new actions of the symmetric group on \(\Bc(n-1)\): every subset \(F^k(\Bc(n-1))\) is equipped with  an action of \(S_{n+k}\). 
\begin{remark}\label{rem:comparing}
We notice that when \(k>2\), if we first embed \(S_n\) into \(S_{n+k}\) in the standard way and then restrict  the  \(S_{n+k}\) action to \(S_n\) we do not obtain the natural \(S_n\) action on \(\Bc(n-1)\). For instance, let us consider \(n=5\) and \(k=3\), and the following nested set \(\MS\) in \(\Bc(4)\):
\[\MS=\{\{1,2\}, \{3,4\}, \{3,4,5\}, \{1,2,3,4,5\}\}\]
On one hand, the  natural action of the transposition \((1,3)\) sends \(\MS\) to 
\[\MS'= \{\{2,3\}, \{1,4\}, \{1,4,5\}, \{1,2,3,4,5\}\}\]
On the other hand,  the partition of \(\{1,2,3,..,8\}\) associated with \(\MS\) is 
\[\{1,2\}, \{3,4\}, \{5,7\}, \{6,8\}  \]
that is sent by the transposition \((1,3)\) to 
\( \{2,3\}, \{1,4\}, \{5,7\}, \{6,8\}.  \)

This last partition corresponds to the nested set 
\[\MS''= \{\{1,4\}, \{2,3\}, \{2,3,5\}, \{1,2,3,4,5\}\}\]
and we notice that \(\MS'\neq \MS''\).


Moreover, we observe that the  natural \(S_5\) action  on \(F^3(\Bc(4))\) and the \(S_5\) action  restricted from \(S_8\)  differ in the number of orbits, therefore   when we consider the associated permutation representations they differ in the multiplicity of the trivial representation, which is 3 for the natural \(S_5\) representation and 4 for the  restricted one.  \end{remark}

\section{The  \(S_{n+k}\) action on the  basis of $H^*(Y_{\Fc_{A_{n-1}}}, \Z)$} \label{ss:action3}

As we observed in the preceding section, the combinatorial action of \(S_{n+k}\) on  \(F^k(\Bc(n-1))\) can  be read as an action on the \(k\)-codimensional strata of $Y_{\Fc_{A_{n-1}}}$. 
Moreover we notice  that this action can in turn be extended to the Yuzvinski basis of $H^*(Y_{\Fc_{A_{n-1}}}, \Z)$ described in Section \ref{constru}.
In fact we can represent the elements of the Yuzvinski basis by {\em labelled partitions} in the way illustrated by the following example.
\begin{example}
\label{esempiobasi}
Let \(n=7\) and let us consider the monomial \(c_{A_1}^2c_{A_2}\) in the Yuzvinski basis of  \(H^6(Y_{\Fc_{A_{6}}}, \Z)\), where \(\{A_1, A_2\}\) is the nested set given by the subspaces \(A_1=\{1,2,3,5\}\), \(A_2=\{4,6,7\}\). Since  \(V\) does not belong to this nested set, we write this monomial as \(c_{A_1}^2c_{A_2}c_V^0\). Now,  according to  the  bijection described in Section \ref{sec:combinatorialaction}, we can associate to the nested set \(\{V, A_1, A_2\}\) the following partition of the set \(\{1,2,..,9\}\):
\[\{1,2,3,5\}\{4,6,7\}\{8,9\}\]
where,  \(A_1\) corresponds to  \(\{1,2,3,5\}\), \(A_2\) corresponds to \(\{4,6,7\}\) and \(V\) corresponds to \(\{8,9\}\).
Finally  we associate to \(c_{A_1}^2c_{A_2}\) the following {\em labelled partition} of \(\{1,2,..,9\}\):
\[\{1,2,3,5\}^2\{4,6,7\}^1\{8,9\}^0\]
As another example,  we represent the monomial \(c_{A_1}^2c_{V}^2\) of \(H^8(Y_{\Fc_{A_{6}}}, \Z)\) by the labelled partition of \(\{1,2,..,8\}\):
\[\{1,2,3,5\}^2\{4,6,7,8\}^2\]
\end{example}
We notice that this representation provides  us with an easy way to `read' the bounds  
 for the exponents in the Yuzvinski basis (see the end of Section \ref{constru}). More in detail, the bounds \(d_{(supp\, f)_A,A}^{\{V\}}\) can be translated in this language in the following way. Let \(I\) be a part of  a  labelled partition of \(\{1,....,n+k\}\)  that represents a monomial in the Yuzvinski basis: then  the   exponent (i.e. the label)  \(\alpha_I\) of   \(I\)  satisfies \(0\leq \alpha_I \leq |I|-2\). 
 Moreover, it  may be equal to 0 only if \(I\) contains the number \(n+k\), i.e. when \(I\) represents \(V\), and in the monomial the variable \(c_V\) does not appear (that is, according to the convention established before, it  appears with exponent 0). In particular all the sets in the partition have cardinality \(\geq 3\) except possibly for the set containing \(n+k\), that may have cardinality equal to 2.
 
Now we observe that  \(S_{n+k}\) acts on the labelled partitions of \(\{1,....,n+k\}\) into \(k+1\) parts, and this provides us with a permutation  action on the monomials of the Yuzvinski basis of $H^*(Y_{\Fc_{A_{n-1}}}, \Z)$.
More in detail:  
\begin{itemize}
\item \(S_{n+k}\) acts on the set of all the monomials that are represented by a labelled partition of \(\{1,....,n+k\}\) into \(k+1\) parts  with  all the labels \(> 0\).  
\item \(S_{n+k-1}\) acts on the set of all the monomials that are represented by a labelled partition of \(\{1,....,n+k\}\) into \(k+1\) parts with  one of  the labels  equal to  \(  0\). In fact if  there is a part labelled by 0, it must contain the number \(n+k\), and \(S_{n+k-1}\) is embedded into  \(S_{n+k}\) as the subgroup that keeps \(n+k\) fixed.
\end{itemize}

This representation, once restricted in the standard way to $S_n$, is not isomorphic to the natural $S_n$ 
representation\footnote{One can see this for instance by counting the multiplicity of the trivial representation in  \(H^6(Y_{A_7})\). The key point is provided by the monomials of type \(c_{A_1}^1c_{A_2}^1c_{A_3}^1c_V^0\) that span an invariant subspace \(H\) for both the natural \(S_8\) action and the extended \(S_{10}\) action.  By an argument similar to that of Remark \ref{rem:comparing}, 
i.e. by counting the number of orbits, one can check that on \(H\) the natural \(S_8\) representation and the \(S_8\) representation restricted from \(S_{10}\) differ  in the  the multiplicity of the trivial representation (that is respectively 3 and 4). }.
Therefore it  is also not compatible with  the action on cohomology produced by the extended geometric \(S_{n+1}\) action described in Section \ref{sec:action1}.
Nevertheless it is interesting since it  splits  the  cohomology module  into a sum of   induced representations of the form  \(Ind_G^{S_{n+k}}Id\) or \(Ind_G^{S_{n+k-1}}Id\)  where \(G\) is the stabilizer of a partition of \(\{1,....,n+k\}\).

Moreover, the  orbits of this  action can  be used to write  a generating formula  for the Poincar\'e polynomials of the models \(Y_{\Fc_{A_{n-1}}}\) that is different from the recursive formula for the Poincar\'e series recalled at the beginning of Section \ref{sec:poincaresupermodels}.

Let us denote by \(\Psi(q,t,z)\)  the following exponential generating series:
\[\Psi(q,t,z)= 1+ \sum_{n\geq 2, \; \MS\in \Nc(\Fc_{A_{n-1}})}P(S) z^{|\MS|} \frac{t^{n+|\MS|-1}}{(n+|\MS|-1)!}\]
where, for every \(n\geq 2\), 
\begin{itemize}
\item  \(\MS\) ranges over all the nested sets of the building set \(\Fc_{A_{n-1}}\) (i.e., \(\MS\) may not contain \(\{V\}\));  
\item  \(P(\MS)\) is the polynomial, in the variable \(q\),  that expresses the contribution to $H^*(Y_{\Fc_{A_{n-1}}}, \Z)$ provided by all the monomials \(m_f\) in the Yuzvinski basis such that \(supp \ f=\MS\). 
For instance, with reference to the Example \ref{esempiobasi}, if \(\MS\) is the nested set \(\{A_1,A_2\}\), then \(P(\MS)=(q+q^2)q\) since we have to take into account all the possible ways to label the partition
\[ \{1,2,3,5\}\{4,6,7\}\{8,9\},\]
while if \(\MS\) is \(\{A_1, V\}\) then \(P(\MS)=(q+q^2)^2\) since we are dealing with the possible labellings of  the partition
\[\{1,2,3,5\}\{4,6,7,8\}.\]
\end{itemize}
We observe that the series \(\Psi(q,t,z)\) encodes the same information that is encoded by the Poincar\'e series. In particular, for a fixed \(n\), the Poincar\'e polynomial  of the model \(Y_{\Fc_{A_{n-1}}}\) can be read from the coefficients of the  monomials whose \(z,t\) component is \(t^kz^s\) with \(k-s=n-1\) (see the Example \ref{esempiopolpoincare} at the end of this section).
 \begin{theorem}
 \label{teo:formulapoicareminimale}
 We have the following formula for the series \(\Psi(q,t,z)\): 
 \begin{equation}
 \label{formulapoicareminimale} \Psi(q,t,z) = e^t\prod_{i\geq 3}e^{zq[i-2]_q\frac{t^i}{i!}}  
 \end{equation}
 where \([j]_q\) denotes the \(q\)-analog of \(j\): \([j]_q=1+q+\cdots+q^{j-1}\).

 \end{theorem}
 \begin{proof}
 We think of the monomials of the Yuzvinsky bases as labelled partitions. Then  we single out the contribution given to \(\Psi\) by all the parts    represented by subsets with cardinality \(i\geq 3\) and with non trivial label. 
  If in a partition there is only one such part  its contribution is \(z(q+q^2+\cdots+ q^{i-2})\frac{t^i}{i!}\), if there are \(j\) such parts their contribution is \(z^j(q+q^2+\cdots+ q^{i-2})^j \frac{(\frac{t^i}{i!})^j}{j!}\). 
  In conclusion the contribution  of all the parts    represented by subsets with cardinality \(i\geq 3\) and with non trivial label is provided by
 \[e^{z(q+q^2+\cdots+ q^{i-2})\frac{t^i}{i!}}-1\]

 Let us now focus on the contribution to \(\Psi\) that comes  from the parts with cardinality \(i\geq 2\) and with label equal to 0. For every monomial in the basis there is  at most one such  part, and its contribution is \(\frac{t^{i-1}}{(i-1)!}\).  The exponent \(i-1\) (instead of  \(i\)) takes into account that such part does not contribute to the cardinality \(|\MS|\).  
The total contribution of the elements with label equal to 0 is therefore     \(\sum_{i\geq 2}\frac{t^{i-1}}{(i-1)!}\). Summing up, we observe that the expression  
  \[ e^t\prod_{i\geq 3}e^{zq[i-2]_q\frac{t^i}{i!}}  \]
  allows us to take into account the contribution to \(\Psi\) of all the possible monomials in the Yuzvinski bases.
   \end{proof}

  \begin{example}
 \label{esempiopolpoincare}
 If one wants to compute the Poincar\'e polynomial of \(Y_{\Fc_{A_{4}}}\) one has to single out all the monomials in \(\Psi\) whose  \(z,t\) component is \(t^kz^s\) with \(k-s=4\).
 A product of the exponential functions that appear in the formula (\ref{formulapoicareminimale}) gives:
 \[\frac{t^4}{4!} [1]+ \frac{t^5}{5!}z [16q+6q^2+q^3]+ \frac{t^6}{6!}z^2 [10q^2]\]
Therefore the  Poincar\'e polynomial is \(1+16q+16q^2+q^3\).
 
 \end{example}
 


\addcontentsline{toc}{section}{References}
\bibliographystyle{amsbib} 
\bibliography{Bibliogpre} 
\end{document}